\documentclass[letterpaper, 10 pt, conference]{ieeeconf}  

\IEEEoverridecommandlockouts                              
\usepackage{amsmath}
\usepackage{amssymb,bm} 
\usepackage{theoremref}
\usepackage{apptools}
\usepackage{graphicx}
\usepackage[caption=false, font=footnotesize]{subfig}
\newtheorem{thm}{Theorem}[section]

\newtheorem{asmp}{Assumption}[section]

\newtheorem{exmp}{Example}[section]
\newtheorem{rem}{Remark}

\title{\LARGE \bf Recursive Averaging with Application to Bio-Inspired 3D Source Seeking}

\author{Mahmoud Abdelgalil and Haithem Taha
\thanks{This work was supported by NSF Grant}
\thanks{M. Abdelgalil is with the Department of Mechanical and Aerospace Engineering,
        University of California Irvine, Irvine, CA 92617, USA
        {\tt\small maabdelg@uci.edu}}%
\thanks{H. Taha is with the Faculty of Mechanical and Aerospace Engineering, University of California Irvine,
        Irvine, CA 92617, USA
        {\tt\small hetaha@uci.edu}}%
}

\begin{document}
\maketitle
\thispagestyle{empty}
\pagestyle{empty}

\begin{abstract}
We analyze a class of high-amplitude, high-frequency oscillatory systems in which periodicity occurs on two distinct time scales and establish the convergence of its trajectories to a suitably averaged system by recursively applying the averaging theorem. Moreover, we introduce a novel bio-inspired 3D source seeking algorithm for rigid bodies with a collocated sensor and prove its practical stability under typical assumptions on the source signal strength field by combining our averaging results with singular perturbation. 
\end{abstract}
\section{Introduction}
The method of averaging has a rich history going back as far as the early works of Laplace and Lagrange on celestial mechanics during the 1700s. Over the course of three centuries, the method continued to evolve thanks to the contributions of many great mathematicians \cite[Appendix A]{sanders2007averaging}. In particular, the subject of averaging received strong interest from Soviet Union mathematicians, which led to the formulation of the Krylov-Bogoliubov-Mitropolskii (KBM) averaging method with application to nonlinear oscillations in physics and engineering \cite{bogoliubov1961}. The emphasis was on constructing successive higher order asymptotically accurate solutions to nonlinear time varying differential equations in the presence of weak oscillations. This task was accomplished by employing the so called \textit{Lie transforms}, also known as \textit{near identity transforms} \cite[Section 3.2]{sanders2007averaging}. Naturally, such a pursuit intertwined the method of averaging with other asymptotic approximation techniques such as the method of multiple-scales, see Nayfeh's book \cite[Chapter 6]{nayfeh2008perturbation}.

Not long after, the \textit{chronological calculus} was developed to provide a representation for the flow of time varying vector fields as an exponential-like series \cite{agrachev1978exponential}. The framework of chronological calculus naturally lends itself to averaging analysis. For example, Bullo utilized the chronological calculus framework in the averaging analysis and vibrational stabilization of mechanical systems \cite{bullo2001series,bullo2002averaging}. Concurrently, the framework was employed by Sarychev \cite{sarychev2001lie} and later Vela \cite{vela2003ageneral} as a geometric formulation of the standard averaging theorem, and, in combination with nonlinear Floquet theory, as a tool for the stability analysis of nonlinear time periodic systems For more details on the connection between the chronological calculus approach to averaging and the KBM method, we refer the reader to the recent article \cite{maggia2020higher}. 

Independently, Sussmann and Liu  \cite{sussmann1991limits,liu1997averaging,liu1997approximation} investigated the use of high-frequency, high-amplitude, periodic signals for motion planning of control-affine systems. Their techniques are based on the earlier work of Kurzweil and Jarnik \cite{kurzweil1987limit} on the limits of solutions of sequences of systems of ordinary differential equations. Notably, Sussmann and Liu's techniques on trajectory approximation and tracking may be combined with the notion of \textit{practical stability} \cite{teel1999semi, moreau2000practical} to analyze the long time behavior of time varying nonlinear systems. This approach is based on establishing the so-called `convergence-of-trajectories' property between the system in consideration and the `extended system' in which the high-frequency, high amplitude-oscillation is replaced by accounting for its average effect on the trajectories. Then, the stability properties of the averaged system are transferred to the practical stability for the original system. In a recent effort, the approach was utilized in analyzing extremum seeking systems with the introduction of the Lie Bracket Approximation framework \cite{durr2013lie,scheinker2014non,grushkovskaya2018class,abdelgalil2021lie}.

In this manuscript, we introduce a new class of high-amplitude, high-frequency oscillatory systems with two distinct fast periodic time scales. This class of systems does not fit within the framework of Sussmann and Liu because of its multiple time scale nature. We establish the convergence-of-trajectories property of this new class of systems to the trajectories of a suitably defined averaged system by recursively applying the higher order periodic averaging theorem for systems with slow time dependence \cite[Section 3.3]{sanders2007averaging}, which allows inferring practical stability results of the original system if the recursively averaged system has a globally uniformly asymptotically stable compact subset. As an application, we propose a novel bio-inspired 3D source seeking algorithm for rigid bodies with a collocated sensor and establish its singular practical asymptotic stability \cite{durr2015singularly} under typical assumptions on the signal strength field using a singularly perturbed version of the theorems we state here.

\section{Second Order Recursive Averaging}
Consider the class of systems on the form:
\begin{align}\label{eq:gen_orig_sys}
    \dot{\textbf{x}} &= \sqrt{\omega}~ \textbf{f}_1(\textbf{x},t,\sqrt{\omega}t, \omega t)+\textbf{f}_2(\textbf{x},t,\sqrt{\omega}t, \omega t)
\end{align}
with an initial condition $\textbf{x}(t_0) = \textbf{x}_0$, where $\textbf{x}$ and $\textbf{x}_0\in\mathbb{R}^n,~t_0\in\mathbb{R},~\omega>0$, and suppose that the following assumption is satisfied: \\
\begin{asmp}\thlabel{asmp:A1}
The time varying vector fields $\textbf{f}_i,~i\in\{1,2\}$ are such that:
\begin{itemize}
    \item[A1] $\textbf{f}_i(\cdot,\cdot,\cdot,\tau_2)\in C^{3-i}(\mathbb{R}^{n+2};\mathbb{R}^n)$, $\forall \tau_2\in\mathbb{R}$
    \item[A2] $\textbf{f}_i(\cdot,\cdot,\cdot,\cdot)\in C^0(\mathbb{R}^{n+3};\mathbb{R}^n)$
    \item[A3] $\textbf{f}_i$ is uniformly bounded in its second argument
    \item[A4] $\exists T_1>0$ s.t. $\textbf{f}_i(\cdot,\cdot,\tau_1+T_1,\cdot)=\textbf{f}_i(\cdot,\cdot,\tau_1,\cdot)$ $\forall \tau_1\in\mathbb{R}$
    \item[A5] $\exists T_2>0$ s.t. $\textbf{f}_i(\cdot,\cdot,\cdot,\tau_2+T_2)=\textbf{f}_i(\cdot,\cdot,\cdot,\tau_2)$ $\forall \tau_2\in\mathbb{R}$
    \item[A6] $\int_0^{T_1}\textbf{f}_1(\cdot,\cdot,\cdot,\tau_2)d\tau_2 =0$\\
\end{itemize}
\end{asmp}
Note that these assumptions may not be minimal but they are typical in the standard averaging literature \cite{ellison1990improved}. In addition, it is clear that this class of systems does not fit within the framework of Sussmann and Liu because of the presence of the intermediate time scale $\sqrt{\omega} t$. 
\begin{rem}
It is not clear from the literature whether Sussmann and Liu's results fit in the same picture or is fundamentally different from the two general approaches to averaging, i.e. the KBM method and the chronological calculus. In fact, some researchers believe that the class of systems that are considered in Sussmann and Liu's work are not directly amenable for averaging analysis by the KBM method (e.g. \cite[Section 5]{durr2013lie}). However, we believe that there is no essential difference between the three approaches.
\end{rem}
\begin{rem}
We note that, aside from some technical regularity conditions on the time dependence, the new class of systems includes as a special case the class of systems considered in the Lie Bracket Approximation framework and the second order case in Sussmann and Liu's framework when the dependence on the time scale $\sqrt{\omega}t$ is trivial.
\end{rem}

Now, consider the averaged system:
\begin{align}\label{eq:gen_avg_sys}
    \dot{\overline{\overline{\textbf{x}}}} &= \overline{\textbf{f}}_1(\overline{\overline{\textbf{x}}},t) + \overline{\textbf{f}}_2(\overline{\overline{\textbf{x}}},t), & \overline{\overline{\textbf{x}}}(t_0) &=\textbf{x}_0
\end{align}
where the time varying vector fields $\overline{\textbf{f}}_1$ and $\overline{\textbf{f}}_2$ are given by:
\begin{alignat}{2}
    \nonumber \overline{\textbf{f}}_1(\textbf{x},t)&= \frac{1}{2T_1 T_2}\int_0^{T_1}\int_0^{T_2}\Big[ \int_0^{\tau_2}&&\textbf{f}_1(\textbf{x},t,\tau_1, s_2) ds_2,  \\
    & & &\textbf{f}_1(\textbf{x},t,\tau_1, \tau_2)\Big]d\tau_2 ~d\tau_1 \\
    \overline{\textbf{f}}_2(\textbf{x},t) &= \frac{1}{T_1 T_2}\int_0^{T_1}\int_0^{T_2} \textbf{f}_2(\textbf{x},t,&&\tau_1, \tau_2)d\tau_2 ~d\tau_1
\end{alignat}
Then, we have the following theorem concerning the relation between the trajectories of the systems (\ref{eq:gen_orig_sys}) and (\ref{eq:gen_avg_sys}):
\begin{thm}\thlabel{thm:T1}
Let a nonempty compact subset $\mathcal{K}\subset\mathbb{R}^n$ and a final time $t_f>0$ be such that a unique trajectory $\overline{\overline{\textbf{x}}}:[t_0,t_0+t_f]\ni t\mapsto\overline{\overline{\textbf{x}}}(t)\in \mathbb{R}^n$ of the system (\ref{eq:gen_avg_sys}) exists $\forall \textbf{x}_0\in \mathcal{K},\,\forall t_0\in \mathbb{R}$. Then $\exists C,\omega_0\in(0,\infty)$ such that $\forall \omega\in(\omega_0,\infty),\, \forall \textbf{x}_0\in \mathcal{K},\, \forall t_0\in\mathbb{R}$, a unique trajectory of the system (\ref{eq:gen_orig_sys}) exists and satisfies:
\begin{align}
 \lVert\textbf{x}(t)-\overline{\overline{\textbf{x}}}(t)\lVert &\leq C/\sqrt{\omega}, & \forall t&\in[t_0,t_0+t_f]\\\nonumber
\end{align}
\end{thm}
\begin{proof} The idea of the proof hinges on a two step averaging procedure for trajectory approximation in which the first step is second order averaging of the system (\ref{eq:gen_orig_sys}) on the time scale $\omega t$, followed by first order averaging of the resulting system on the time scale $\sqrt{\omega} t$; hence the `recursive' nature. This approach may be generalized to higher orders in a similar fashion. 
First, we apply the time scaling $\tau = \omega (t-t_0)$, and let $\varepsilon = \frac{1}{\sqrt{\omega}}$ to obtain the system:
\begin{align}\label{eq:avging_step_0}
    \frac{d\textbf{x}}{d\tau} &=  \varepsilon\, \textbf{f}_1(\textbf{x},\varepsilon^2\tau,\varepsilon \tau, \tau) + \varepsilon^2\textbf{f}_2(\textbf{x},\varepsilon^2\tau,\varepsilon \tau, \tau)
\end{align}
which is on the averaging canonical form. Note that we suppressed the dependency on the initial time $t_0$ for brevity, but it is implied. By applying the stroboscopic averaging procedure for systems with slow time dependence to second order in $\varepsilon$ \cite[Section 3.3]{sanders2007averaging}, we obtain the system:
\begin{equation}
\begin{aligned}    \label{eq:avging_step_1_1}
\frac{d\overline{\textbf{x}}}{d\tau}=  \varepsilon^2 \Bigg(\frac{1}{T_2}\int_0^{T_2}\textbf{f}_2(\overline{\textbf{x}},\varepsilon^2\tau,\varepsilon &\tau, \tau_2)d\tau_2 ~+  \\
    \frac{1}{2T_2}\int_0^{T_2}\bigg[ \int_0^{\tau_2}\textbf{f}_1(&\overline{\textbf{x}},\varepsilon^2\tau,\varepsilon\tau, s_2) ds_2, \\
    &\textbf{f}_1(\overline{\textbf{x}},\varepsilon^2\tau,\varepsilon\tau, \tau_2)\bigg]d\tau_2\Bigg)
\end{aligned}
\end{equation}
A second time scale change to $\sigma = \varepsilon\tau$ leads to the system:
\begin{equation}
\begin{aligned}\label{eq:avging_step_1_2}
    \frac{d\overline{\textbf{x}}}{d\sigma}=  \varepsilon \bigg(\frac{1}{T_2}\int_0^{T_2}\textbf{f}_2(\overline{\textbf{x}},\varepsilon\sigma,\sigma, \tau_2&)d\tau_2 +  \\
    \frac{1}{2T_2}\int_0^{T_2}\Bigg[ \int_0^{\tau_2}\textbf{f}_1(&\overline{\textbf{x}},\varepsilon\sigma,\sigma, s_2) ds_2,\\
    &\textbf{f}_1(\overline{\textbf{x}},\varepsilon\sigma,\sigma, \tau_2)\Bigg]d\tau_2\bigg)
\end{aligned}
\end{equation}
which is again on the averaging canonical form. By applying the stroboscopic periodic averaging procedure for systems with slow time dependence to first order in $\varepsilon$ \cite[Section 3.3]{sanders2007averaging}, we obtain the system:
\begin{align}
    \frac{d\overline{\overline{\textbf{x}}}}{d\sigma} &= \varepsilon\left(\overline{\textbf{f}}_1(\overline{\overline{\textbf{x}}},\varepsilon\sigma) + \overline{\textbf{f}}_2(\overline{\overline{\textbf{x}}},\varepsilon\sigma)\right)
\end{align}
A final time scale change to $t=\varepsilon \sigma$ brings the system to the fully averaged form:
\begin{align}
    \label{eq:avging_step_2}
    \dot{\overline{\overline{\textbf{x}}}} &= \overline{\textbf{f}}_1(\overline{\overline{\textbf{x}}},t) + \overline{\textbf{f}}_2(\overline{\overline{\textbf{x}}},t), & \overline{\overline{\textbf{x}}}(t_0)&= \textbf{x}_0
\end{align}
From the assumptions of the theorem, we know that $\forall \textbf{x}_0\in \mathcal{K},~\forall t_0\in \mathbb{R}$, a unique trajectory $\overline{\overline{\textbf{x}}}(t)$ of the system (\ref{eq:avging_step_2}) exists on the compact time interval $t\in [t_0,t_0+t_f]$. Moreover. we know that the vector fields $\textbf{f}_i$ are uniformly bounded in the second argument, which implies that there exists a compact subset $\mathcal{M}\subset\mathbb{R}^n$ such that $\overline{\overline{\textbf{x}}}(t)\in \mathcal{M},\,\forall t\in[t_0,t_0+t_f],\,\forall \textbf{x}_0\in \mathcal{K}, \,\forall t_0\in \mathbb{R}$. Hence, the first order periodic averaging theorem \cite{sanders2007averaging} ensures the existence of $\varepsilon_1,C_1\in(0,\infty)$ such that $\forall\varepsilon\in(0,\varepsilon_1),\,\forall \textbf{x}_0\in \mathcal{K},\,\forall t_0\in\mathbb{R}$, a unique trajectory $\overline{\textbf{x}}(\sigma)$ of the system (\ref{eq:avging_step_1_2}) exists on the time interval $\sigma\in \left[0, t_f/\varepsilon\right]$ and satisfies:
\begin{align}
    \lVert \overline{\overline{\textbf{x}}}(\varepsilon\sigma)-\overline{\textbf{x}}(\sigma)\lVert &\leq C_1\varepsilon, & \forall \sigma&\in \left[0,t_f/\varepsilon\right]
\end{align}
Equivalently, a unique trajectory $\overline{\textbf{x}}(\varepsilon\tau)$ of the system (\ref{eq:avging_step_1_1}) exists on the compact time interval $\tau\in \left[0,t_f/\varepsilon^2\right]$ and $\overline{\textbf{x}}(\varepsilon\tau)\in \mathcal{M}_{\varepsilon_1}$, where the compact subset $\mathcal{M}_{\varepsilon_1}$ is defined by:
\begin{align}
   \mathcal{M}_{\varepsilon_1}&= \left\{\overline{\textbf{x}}\in\mathbb{R}^n~\left|~\inf_{\textbf{x}\in \mathcal{K}}{\lVert \textbf{x}-\overline{\textbf{x}}\lVert}\leq C_1 \varepsilon_1 \right.\right\}
\end{align}
Hence, the conditions of the second order periodic averaging theorem with trade-off \cite[Section 2.9]{sanders2007averaging}, \cite{murdock1983some} are satisfied and we are guaranteed the existence of $\varepsilon_2\in(0,\infty)$ such that $\forall\varepsilon\in(0,\varepsilon_2),\,\forall \textbf{x}_0\in \mathcal{K},\,\forall t_0\in \mathbb{R}$ a unique trajectory $\textbf{x}(\tau)$ of the system (\ref{eq:avging_step_0}) exists on the compact time interval $\tau\in\left[0,t_f/\varepsilon^2\right]$ and satisfies: 
\begin{align}
    \lVert \overline{\textbf{x}}(\varepsilon\tau)-\textbf{x}(\tau)\lVert &\leq C_2\varepsilon, & \forall \tau&\in \left[0,t_f/\varepsilon^2\right]
\end{align}
Using the fact that $\varepsilon = 1/\sqrt{\omega}$, it follows from the triangle inequality that $\forall\omega \in(\omega_0,\infty)$:
\begin{align}
    \left\lVert\overline{\overline{\textbf{x}}}(\varepsilon^2\tau)-\textbf{x}(\tau)\right\lVert &\leq C/\sqrt{\omega}, & \forall \tau&\in[0,\omega\,t_f]
\end{align}
where $C=C_1+C_2$, and $\omega_0=1/\varepsilon_2^2$. With some healthy notation abuse, it follows that:
\begin{align}
    \left\lVert\overline{\overline{\textbf{x}}}(t)-\textbf{x}(t)\right\lVert &\leq C/\sqrt{\omega}, & \forall t&\in[t_0,t_0+t_f]
\end{align}
\end{proof}
\begin{rem}
Observe that we have not made any mention of the near-identity transforms in the proof nor in the construction of the averaged system, despite the fact that we employed the higher order averaging theorem. This is due to the fact the $\mathcal{O}(\varepsilon)$ terms in the standard averaging procedure vanish, and so the contribution of the near-identity transform to the accuracy of the solution is of the same order as the remainders in the averaging procedure, i.e. $\mathcal{O}(\varepsilon)$ \cite{murdock1983some}.
\end{rem}
This result establishes the convergence-of-trajectories property, i.e. Hypothesis 2 in \cite{moreau2000practical}, which paves the way for transferring stability properties from the averaged system (\ref{eq:gen_avg_sys}) to the original system (\ref{eq:gen_orig_sys}). In particular, the proof of the following theorem is straightforward:
\begin{thm}\thlabel{thm:T2}
Suppose that a compact subset $\mathcal{S}\subset \mathbb{R}^n$ is globally uniformly asymptotically stable for the system (\ref{eq:gen_avg_sys}). Then the subset $\mathcal{S}$ is semi-globally practically uniformly asymptotically stable for the system (\ref{eq:gen_orig_sys}).  
\end{thm}
\section{Singularly Perturbed Recursive Averaging}
The results of the previous section may be combined with singular perturbation techniques to expand the class of systems considered. In particular, consider the class of systems on the form:
\begin{gather}\label{eq:gen_sing_orig_sys_1}
    \dot{\textbf{x}} =\sqrt{\omega}~ \textbf{f}_1(\textbf{x},t,\sqrt{\omega}t, \omega t)+ \textbf{f}_2(\textbf{x},t,\sqrt{\omega}t, \omega t),\\
    \label{eq:gen_sing_orig_sys_2}
    \mu ~\dot{\textbf{z}} = \textbf{g}(\textbf{x},\textbf{z}),
\end{gather}
subject to the initial conditions $\textbf{x}(t_0)=\textbf{x}_0\in\mathbb{R}^n$ and $\textbf{z}(t_0)=\textbf{z}_0\in\mathbb{R}^m$, where $\mu,\omega>0$. Suppose the following assumption is satisfied:
\begin{asmp}\thlabel{asmp:A2}
The vector field $\textbf{g}$ is such that
\begin{itemize}
    \item[B1] $\textbf{g}(\cdot,\cdot)\in C^1(\mathbb{R}^n\times\mathbb{R}^m;\mathbb{R}^m)$
    \item[B2] $\exists!\bm{\varphi}\in C^2(\mathbb{R}^n;\mathbb{R}^m)$ s.t. $\textbf{g}(\textbf{x},\bm{\varphi}(\textbf{x})) = 0, \forall \textbf{x}\in\mathbb{R}^n$
    \item[B3] $\forall\textbf{x}\in\mathbb{R}^n$, the point $\bm{\varphi}(\textbf{x})$ globally uniformly asymptotically stable for the system:
    \begin{align*}
        \dot{\textbf{z}} &= \textbf{g}(\textbf{x},\textbf{z}), &\textbf{z}(t_0)&=\textbf{z}_0 \\
\end{align*}
\end{itemize}
\end{asmp}
Furthermore, consider the reduced order recursively averaged (RORA) system:
\begin{align}\label{eq:gen_sing_avg_sys}
    \dot{\overline{\overline{\textbf{x}}}} &= \overline{\textbf{f}}_1(\overline{\overline{\textbf{x}}},t)+\overline{\textbf{f}}_2(\overline{\overline{\textbf{x}}},t), & \overline{\overline{\textbf{x}}}(t_0) &=\textbf{x}_0
\end{align}
where the time varying vector fields $\overline{\textbf{f}}_1$ and $\overline{\textbf{f}}_2$ are given by:
\begin{align}
    \nonumber\overline{\textbf{f}}_1(\textbf{x},t)= \frac{1}{T_1 T_2}\int_0^{T_1}\int_0^{T_2}\Big[ \int_0^{\tau_2}&\tilde{\textbf{f}}_1(\textbf{x},t,\tau_1, s_2) ds_2,  \\
    &\tilde{\textbf{f}}_1(\textbf{x},t,\tau_1, \tau_2)\Big]d\tau_2 ~d\tau_1 \\
    \overline{\textbf{f}}_2(\textbf{x},t) = \frac{1}{2T_1 T_2}\int_0^{T_1}\int_0^{T_2} \tilde{\textbf{f}}_2(\textbf{x},&t,\tau_1, \tau_2)d\tau_2 ~d\tau_1
\end{align}
and the time varying vector fields $\tilde{\textbf{f}}_1$ and $\tilde{\textbf{f}}_2$ are given by:
\begin{align}
    \tilde{\textbf{f}}_1(\textbf{x},t,\cdot,\cdot) &= \textbf{f}_1(\textbf{x},\bm{\varphi}(\textbf{x}),t,\cdot,\cdot)\\
    \tilde{\textbf{f}}_2(\textbf{x},t,\cdot,\cdot) &= \textbf{f}_2(\textbf{x},\bm{\varphi}(\textbf{x}),t,\cdot,\cdot)
\end{align}
Then, we have the following theorem concerning the relation between the stability properties of the system (\ref{eq:gen_sing_orig_sys_1})-(\ref{eq:gen_sing_orig_sys_2}) and the RORA system (\ref{eq:gen_sing_avg_sys}):
\begin{thm}\thlabel{thm:T3}
Suppose that a compact subset $\mathcal{S}\subset \mathbb{R}^n$ is globally uniformly asymptotically stable for the RORA system (\ref{eq:gen_sing_avg_sys}). Then the subset $\mathcal{S}$ is singularly semi-globally practically uniformly asymptotically stable for the system (\ref{eq:gen_sing_orig_sys_1})-(\ref{eq:gen_sing_orig_sys_2}).
\end{thm}
\begin{proof}
The proof of the theorem follows along the same lines in \cite{durr2015singularly} except for the replacement of the Lie Bracket approximation theorems introduced in \cite{durr2013lie} with the recursive averaging approximation procedure we introduced in \thref{thm:T1} and \thref{thm:T2}.
\end{proof}
\section{3D Source Seeking for Rigid Bodies}
Source seeking is the problem of locating a target that emits a scalar measurable signal, typically without global positioning information \cite{cochran20093}. Interestingly, the taxis of microorganisms can be framed as a source seeking problem in which the seeking agent is the microorganism, and the signal may be light (phototaxis), chemical concentration (chemotaxis), or even temperature (thermotaxis). It turns out that millions of years of evolution led certain microorganisms to utilize extremum seeking based algorithms to solve the source seeking problem \cite{abdelgalil2021sperm}. This finding attests to the robustness and simplicity of extremum seeking control. Moreover, the finding serves as an invitation to discover new bio-inspired source seeking algorithms by mimicking nature. This section is a step in such direction. 

The 3D kinematics of a rigid body are given by
\begin{align}
    \label{eq:src_sys_1}
    \dot{\textbf{p}} &= \textbf{R}\textbf{v} \\
    \label{eq:src_sys_2}
    \dot{\textbf{R}} &= \textbf{R}\widehat{\bm{\Omega}}
\end{align}
where $\textbf{p}\in\mathbb{R}^3$ is the position of the origin of the body frame with respect to the reference frame, \textbf{v} is the velocity in body coordinates, $\textbf{R}\in \text{SO}(3)$ is the rotation matrix that relates the body frame to the reference frame, and $\bm{\Omega}\in\mathbb{R}^3$ is the angular velocity in body coordinates. The map $\widehat{\bullet}:\mathbb{R}^3\rightarrow\mathbb{R}^{3\times 3}$ takes a vector $\bm{\Omega}=\left[\Omega_1,\Omega_2,\Omega_3\right]^\intercal \in \mathbb{R}^3$ to the corresponding skew symmetric matrix, and it has the property that $\widehat{\textbf{R}\bm{\Omega}}= \textbf{R}\widehat{\bm{\Omega}}\textbf{R}^\intercal$ when $\textbf{R}$ is a rotation matrix. We consider a vehicle model in which the velocity $\textbf{v}$ and angular velocity $\bm{\Omega}$ are given by:
\begin{align}\label{eq:velo_def}
    \textbf{v} &= \sqrt{2\omega}\,\textbf{e}_1, & \bm{\Omega} &= \Omega_\parallel \textbf{e}_1 + \Omega_\perp\textbf{e}_3
\end{align}
where $\Omega_\parallel$ and $\Omega_\perp$ are the control inputs, which represent roll and yaw of the body frame, respectively, and $\omega$ is a positive parameter. 
\begin{rem}
This model is a natural extension of the unicycle model to the 3D setting. It is well known that this system is controllable using depth one Lie brackets \cite{leonard1993averaging}.\\
\end{rem}

Let $c\in C^2(\mathbb{R}^3;\mathbb{R})$ represent the signal strength field and define the control inputs $\Omega_\parallel$ and $\Omega_\perp$ by the dynamic time periodic feedback law:
\begin{align}
    \label{eq:src_law_1}
    \mu\,\dot{z} &= c(\textbf{p})-z\\
    \label{eq:src_law_2}
    \Omega_\perp(c(\textbf{p}),z) &= \omega - \dot{z} \\
    \label{eq:src_law_3}
    \Omega_\parallel(z,\omega t) &= 2\alpha \sqrt{2\omega}\sin\left(\omega t-z+\frac{\pi}{4}\right)
\end{align}
where $\omega>0$,  $\mu$, and $\alpha>0$ are constant parameters.
\begin{rem}
When the motion is confined to the plane (i.e. $\alpha=0$), this control law is the same as the source seeking algorithm for the unicycle model introduced in \cite{scheinker2014extremum}, which also turns out to be the same algorithm employed by sea urchin sperm cells for seeking the egg in 2D \cite{abdelgalil2021sperm}. Here we extend the controller to the 3D setting and establish its practical stability. We emphasize that the 2D source seeking algorithm in \cite{scheinker2014extremum} does not work directly in 3D; the current framework of recursive averaging is needed. We also note that other choices for $\Omega_\parallel$ are possible, particularly ones that do not depend on $z$ at all. That is, we could have chosen $\Omega_\parallel(z,\omega t)=\sqrt{\omega} \sin(\omega t)$, i.e. an open loop controller, that would have still worked. However, the computations would have been slightly more involved.
\end{rem} 

\begin{asmp}\thlabel{asmp:A3}
Suppose that the signal strength field $c\in C^2(\mathbb{R}^3;\mathbb{R})$ is radially unbounded, $\exists!\textbf{p}^*\in \mathbb{R}^3$ such that $\nabla c(\textbf{p})= 0 \iff \textbf{p}=\textbf{p}^*$, and satisfies the inequality:
\begin{align}
c(\textbf{p})-c(\textbf{p}^*)\geq -\kappa \lVert \nabla c(\textbf{p})\lVert^2, \, \forall \textbf{p}\in\mathbb{R}^n\\\nonumber
\end{align}
\end{asmp}
\begin{figure*}[t]
    \centering\hfill\subfloat[Signal strength versus time for \thref{exmp:ex1} (upper) and \thref{exmp:ex2} (lower)]{
    \label{fig:ex_costs}\hfill \includegraphics[width=0.85\columnwidth]{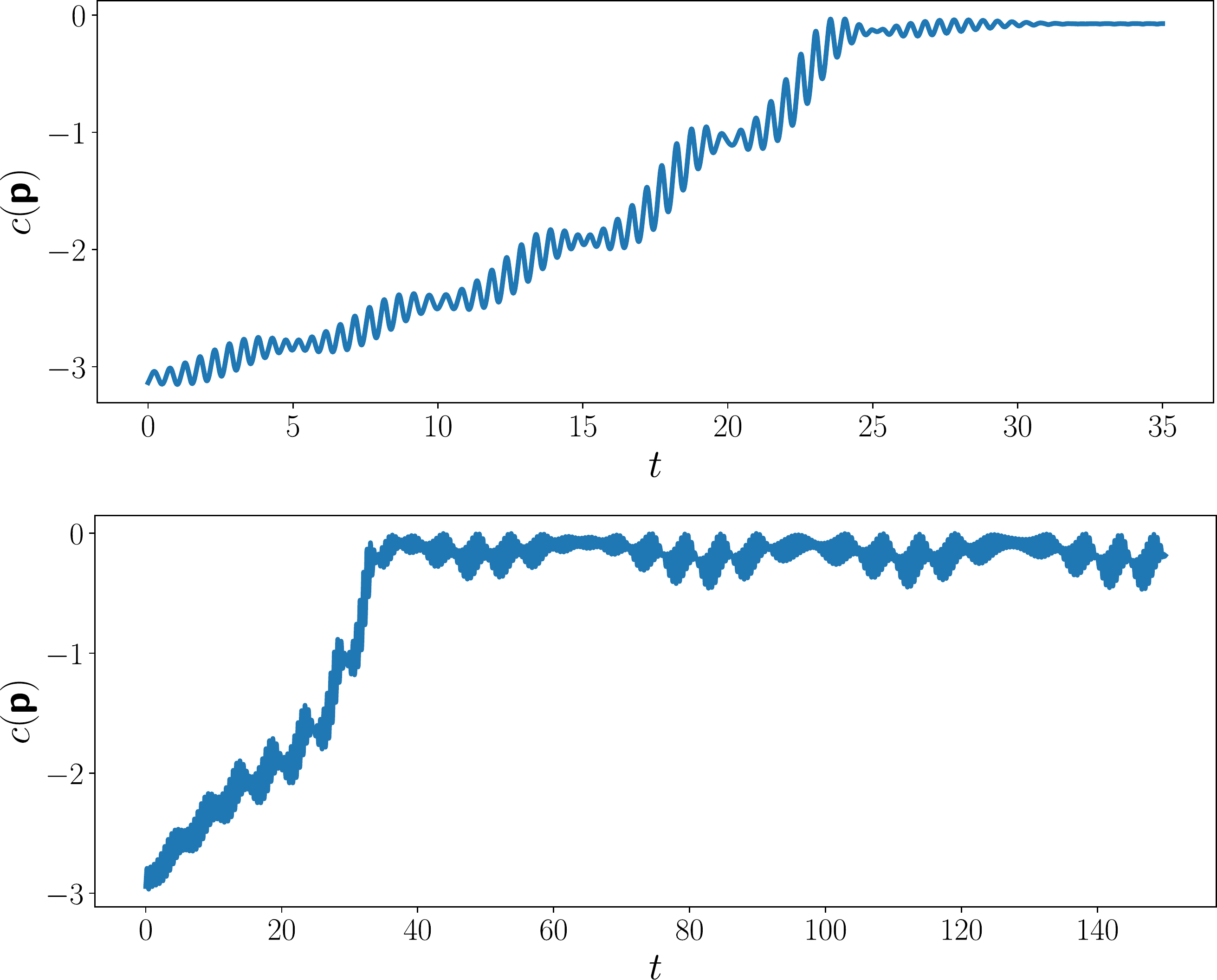}}\hfill \subfloat[3D Trajectory for \thref{exmp:ex1} (left) and \thref{exmp:ex2} (right)]{
    \label{fig:ex_trajs}\hfill \includegraphics[width=1\columnwidth]{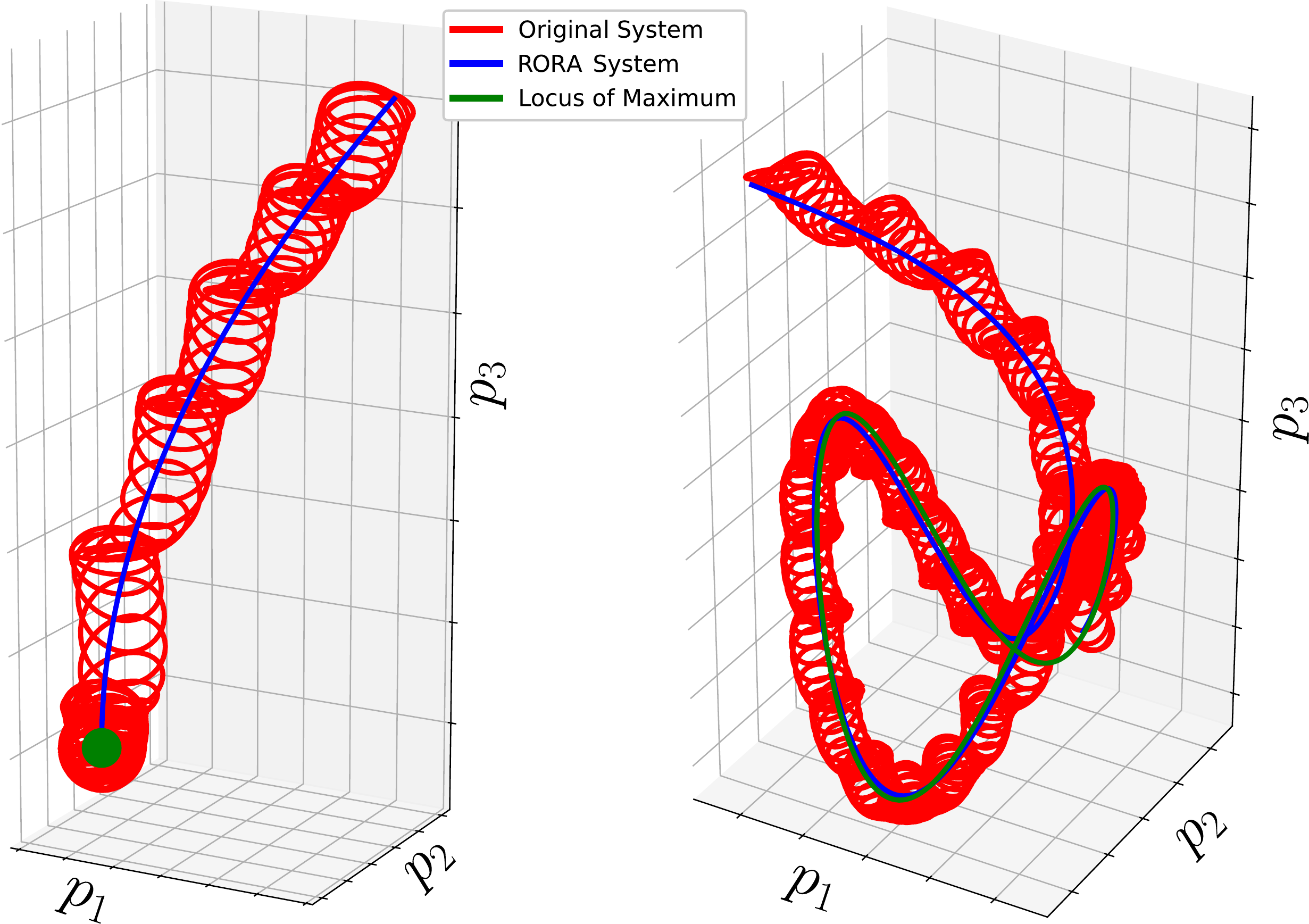}}\hfill
    \caption{Numerical simulation results}
\end{figure*}
Then, we have the following theorem:
\begin{thm}
Let \thref{asmp:A3} be satisfied. Then, the compact subset $\{\textbf{p}^*\}\times\text{SO}(3)$ is singularly semi-globally practically uniformly asymptotically stable for the control system defined by (\ref{eq:src_sys_1})-(\ref{eq:velo_def}) under the dynamic feedback law defined by (\ref{eq:src_law_1})-(\ref{eq:src_law_3}). 
\end{thm}
\begin{proof}
Define $\tau=\omega t,~ \sigma=\sqrt{\omega} t$, and the intermediate rotation $\textbf{Q} = \textbf{R} \textbf{R}_1^\intercal \textbf{R}_2^\intercal$, where: 
\begin{align*}
    \textbf{R}_1&=\text{exp}\left((\tau -z)\widehat{\textbf{e}}_3\right), & \dot{\textbf{R}}_1&=(\omega -\dot{z})\textbf{R}_1 \widehat{\textbf{e}}_3 \\
    \textbf{R}_2 &= \text{exp}\left(\alpha\sigma\left(\widehat{\textbf{e}}_1+\widehat{\textbf{e}}_2\right)\right), & \dot{\textbf{R}}_2&=\alpha\sqrt{\omega}\textbf{R}_2 \left(\widehat{\textbf{e}}_1+\widehat{\textbf{e}}_2\right)
\end{align*}
Then, compute:
\begin{align*}
    \dot{\textbf{Q}} &= \dot{\textbf{R}}\textbf{R}_1^\intercal \textbf{R}_2^\intercal+ \textbf{R}\dot{\textbf{R}}_1^\intercal \textbf{R}_2^\intercal+\textbf{R}\textbf{R}_1^\intercal\dot{\textbf{R}}_2^\intercal
\end{align*}
Observe that:
\begin{align}
    \dot{\textbf{R}}\textbf{R}_1^\intercal \textbf{R}_2^\intercal&= \textbf{Q}\textbf{R}_2 \textbf{R}_1\left(\Omega_\parallel\widehat{\textbf{e}}_1 + \Omega_\perp\widehat{\textbf{e}}_3\right)\textbf{R}_1^\intercal \textbf{R}_2^\intercal \\
    \textbf{R}\dot{\textbf{R}}_1^\intercal \textbf{R}_2^\intercal &= -\textbf{Q}\textbf{R}_2 \textbf{R}_1\left(\Omega_\perp\widehat{\textbf{e}}_3\right)\textbf{R}_1^\intercal \textbf{R}_2^\intercal\\
    \textbf{R}\textbf{R}_1^\intercal\dot{\textbf{R}}_2^\intercal &= -\alpha\sqrt{\omega}\textbf{Q}\textbf{R}_2 \left( \widehat{\textbf{e}}_1+\widehat{\textbf{e}}_2\right)\textbf{R}_2^\intercal
\end{align}
Consequently, we have that:
\begin{align}
    \dot{\textbf{Q}} &= \textbf{Q}\textbf{R}_2\left(\Omega_\parallel\textbf{R}_1\widehat{\textbf{e}}_1\textbf{R}_1^\intercal-\alpha\sqrt{\omega}\left( \widehat{\textbf{e}}_1+\widehat{\textbf{e}}_2\right)\right)\textbf{R}_2^\intercal
\end{align}
Direct computation shows that:
\begin{align*}
    \Omega_\parallel(z,\tau) \textbf{R}_1\textbf{e}_1 &= \alpha\sqrt{\omega}\left(\textbf{e}_1+\textbf{e}_2 + \bm{\Omega}_1(\tau,z)\right)
\end{align*}
where:
\begin{align}
    \bm{\Omega}_1(\tau,z) = \text{exp}\left(2(\tau-z(t))\widehat{\textbf{e}}_3\right)(\textbf{e}_1-\textbf{e}_2)
\end{align}
Hence, we have that:
\begin{align}
    \dot{\textbf{Q}} &=\alpha\sqrt{\omega} \textbf{Q}\textbf{R}_2\widehat{\bm{\Omega}}_1(\tau,z)\textbf{R}_2^\intercal = \sqrt{\omega} \textbf{Q}\widehat{\bm{\Lambda}}(\tau,z)
\end{align} 
where 
\begin{align*}
\bm{\Lambda}(z,\sigma,\tau)= \alpha~\text{exp}\left(\alpha\sigma\left(\widehat{\textbf{e}}_1+\widehat{\textbf{e}}_2\right)\right)\bm{\Omega}_1(\tau,z)
\end{align*}
Moreover, the position kinematics can be rewritten as:
\begin{align}
    \dot{\textbf{p}} &= \textbf{R}\textbf{v} = \textbf{Q} \textbf{R}_2 \textbf{R}_1 \textbf{v} = \sqrt{\omega}\,\textbf{Q} \,\textbf{f}(z,\sigma,\tau)
\end{align} 
where $\textbf{f}$ is given by:
\begin{align*}
    \textbf{f}(z,\sigma,\tau) &=\sqrt{2}\, \text{exp}\left(\alpha\sigma\left(\widehat{\textbf{e}}_1+\widehat{\textbf{e}}_3\right)\right)\text{exp}\left((\tau -z)\widehat{\textbf{e}}_3\right)\textbf{e}_1
\end{align*}
The time evolution of new variables $\textbf{p}$, $\textbf{Q}$ and $z$ is governed by the system:
\begin{align}\label{eq:orig_dynamics_1}
    \dot{\textbf{p}} &= \sqrt{\omega}\,\textbf{Q}\,\textbf{f}(z,\sigma,\tau)\\\label{eq:orig_dynamics_2}
    \dot{\textbf{Q}} &= \sqrt{\omega}\,\textbf{Q}\, \widehat{\bm{\Lambda}}(z,\sigma,\tau)\\\label{eq:orig_dynamics_3}
    \mu\,\dot{z} &= c(\textbf{p})-z
\end{align}
We embed the manifold $\text{SO}(3)$ into $\mathbb{R}^3\times\mathbb{R}^3\times\mathbb{R}^3$ by partitioning the matrix $\textbf{Q}$ into column vectors $\textbf{Q} = \left[\textbf{q}_1,\, \textbf{q}_2,\, \textbf{q}_3\right]$. Observe that:
\begin{align}
    \dot{\textbf{Q}}&= \textbf{Q} \widehat{\bm{\Lambda}} = \textbf{Q} \widehat{\bm{\Lambda}}\textbf{Q}^\intercal \textbf{Q} = \widehat{\textbf{Q}\bm{\Lambda}} \textbf{Q}
\end{align}
and so it is easy to see that the time evolution of the columns of $\textbf{Q}$ is governed by:
\begin{align}
    \frac{d\textbf{q}_j}{dt} &=\sqrt{\omega} \sum_{i=1}^3\Lambda_i(z,\sigma,\tau)\,\textbf{q}_i\times\textbf{q}_j, & j&\in\{1,2,3\}
\end{align}
Next, define the state vector $\textbf{x}\in\mathbb{R}^{12}$ by $\textbf{x} = \left[\textbf{p}^\intercal,\,\textbf{q}_1^\intercal,\, \textbf{q}_2^\intercal,\, \textbf{q}_3^\intercal\right]^\intercal$, and the vector field $\textbf{X}$ which is given in coordinates by:
\begin{align}
    \textbf{X}(\textbf{x},z,\sigma,\tau)&=\left[\begin{array}{c}
        \sum\limits_i^3 f_i(z,\sigma,\tau)\textbf{q}_i \\
        \sum\limits_{i,k}^3 \Lambda_i(z,\sigma,\tau)\epsilon_{i1k}{\textbf{q}}_k \\
        \sum\limits_{i,k}^3 \Lambda_i(z,\sigma,\tau)\epsilon_{i2k}{\textbf{q}}_k \\
        \sum\limits_{i,k}^3 \Lambda_i(z,\sigma,\tau)\epsilon_{i3k}{\textbf{q}}_k
    \end{array}\right] 
\end{align}
where $\epsilon_{ijk}$ is the Levi-Civita symbol. Restrict the initial conditions to lie on the manifold $\mathcal{M}=\{(\textbf{x},z)\in\mathbb{R}^{13}\,|\, \textbf{q}_i^\intercal\textbf{q}_j=\delta_{ij},~\textbf{q}_i\times\textbf{q}_j = \epsilon_{ijk}\textbf{q}_k \}$, where $\delta_{ij}$ is the Kronecker symbol. With this embedding, the kinematics can be written succinctly as:
\begin{align}\label{eq:step0_sys_1}
    \frac{d\textbf{x}}{dt} &=\sqrt{\omega} \textbf{X}(\textbf{x},z,\sigma,\tau)\\
    \label{eq:step0_sys_2} \mu\,\frac{dz}{d\tau} &= c(\textbf{p})-z
\end{align}
Observe that the system is now on a similar form to (\ref{eq:gen_sing_orig_sys_1}) and (\ref{eq:gen_sing_orig_sys_2}) and satisfies \thref{asmp:A1} and \thref{asmp:A2}. Hence, we may apply \thref{thm:T3} to investigate the stability properties of the system. To proceed, we compute the RORA system.

\subsubsection*{Step I} The first step is singular perturbation which yields the reduced order system:
\begin{align}\label{eq:step1_sys_1}
    \frac{d\widetilde{\textbf{x}}}{dt} &=\sqrt{\omega} \widetilde{\textbf{X}}(\widetilde{\textbf{x}},\sigma,\tau)
\end{align}
which evolves on the slow manifold $\widetilde{\mathcal{M}}=\{(\textbf{x},z)\in \mathcal{M} \,| z=c(\textbf{p})\}$
and the vector field $\widetilde{\textbf{X}}$ is given by:
\begin{align}
    \widetilde{\textbf{X}}(\widetilde{\textbf{x}},\sigma,\tau)&= \textbf{X}(\widetilde{\textbf{x}},c(\widetilde{\textbf{p}}),\sigma,\tau)
\end{align}
\subsubsection*{Step II} The next step in the computation is second order periodic averaging for systems with slow time dependence applied to the reduced order system, which yields the vector field $\overline{\textbf{X}}$:
\begin{align}
    \overline{\textbf{X}}(\overline{\textbf{x}},\sigma) =\frac{1}{4\pi} \int_0^{2\pi}\left[\int\widetilde{\textbf{X}}(\overline{\textbf{x}},\sigma,\tau)d\tau, \widetilde{\textbf{X}}(\overline{\textbf{x}},\sigma,\tau)\right]d\tau
\end{align}
which can be computed in coordinates as:
\begin{align*}
    \overline{\textbf{X}}
    &=\left[\begin{array}{c}
        \sum\limits_{i,j}^3 \alpha_{ij}(\sigma) \overline{\textbf{q}}_i\overline{\textbf{q}}_j^\intercal\nabla c(\overline{\textbf{p}})\\
        \sum\limits_{i,k}\left(\beta_{3i}(\sigma)\epsilon_{i2k}+\beta_{i2}(\sigma)\epsilon_{i3k}\right)\overline{\textbf{q}}_k\\ \sum\limits_{i,k}\left(\beta_{i3}(\sigma)\epsilon_{i1k}+\beta_{1i}(\sigma)\epsilon_{i3k}\right)\overline{\textbf{q}}_k\\ \sum\limits_{i,k}\left(\beta_{2i}(\sigma)\epsilon_{i1k}+\beta_{i1}(\sigma)\epsilon_{i2k}\right)\overline{\textbf{q}}_k
    \end{array}\right]
\end{align*}
where the functions $\alpha_{ij}$ and $\beta_{ij}$ are given by:
\begin{align}
    \alpha_{ij}(\sigma)&= \frac{1}{2\pi}\int_0^{2\pi}\left(\frac{\partial f_i}{\partial z}\int f_j d\tau - \int\frac{\partial f_i}{\partial z}d\tau f_j\right)d\tau \\
    \beta_{ij}(\sigma)&= \frac{1}{2\pi}\int_0^{2\pi}\left(\Lambda_i\int \Lambda_j d\tau-\int\Lambda_i d\tau\Lambda_j\right)d\tau
\end{align}
\subsubsection*{Step III} The final computation is the averaging of $\overline{\textbf{X}}$ over its period to obtain the vector field $\overline{\overline{\textbf{X}}}$. The result of this computation is: 
\begin{align*}
    \overline{\overline{\textbf{X}}} &= \left[\begin{array}{c}
        \sum\limits_{i,j}^3 \text{A}_{ij} \overline{\overline{\textbf{q}}}_i\overline{\overline{\textbf{q}}}_j^\intercal\nabla c(\overline{\overline{\textbf{p}}})\\
        \textbf{0}
    \end{array}\right]
\end{align*}
where $\text{A}_{ij}$ are the entries of the matrix given by:
\begin{align*}
    \textbf{A}&= \frac{1}{4}\left[ \begin{array}{ccc}
        3 & 1 & 0 \\
        1 & 3 & 0 \\
        0 & 0 & 2
    \end{array}\right]
\end{align*}
By changing the coordinates back, we obtain:
\begin{align}\label{eq:red_dynamics_1}
    \frac{d\overline{\overline{\textbf{p}}}}{dt}&= \overline{\overline{\textbf{Q}}}\textbf{A}\overline{\overline{\textbf{Q}}}^\intercal \nabla c(\overline{\overline{\textbf{p}}}) , &
    \frac{d\overline{\overline{\textbf{Q}}}}{dt}&=\textbf{0} , &
    \overline{\overline{z}}&=c(\overline{\overline{\textbf{p}}})
\end{align}
The compact subset $\mathcal{S}= \{\textbf{x}\in \widetilde{\mathcal{M}}~|~\textbf{p}=\textbf{p}^*\}$ 
is globally uniformly asymptotically stable for the dynamics defined by (\ref{eq:red_dynamics_1}). This is easy to see since the matrix $\overline{\overline{\textbf{Q}}}\textbf{A}\overline{\overline{\textbf{Q}}}^\intercal$ is positive definite $\forall \overline{\overline{\textbf{Q}}}\in\text{SO}(3)$, and $\overline{\overline{\textbf{Q}}}$ does not change. Reverting back to $\overline{\overline{\textbf{R}}}$ from $\overline{\overline{\textbf{Q}}}$ will not affect this stability result.
\end{proof}
\section{Numerical Simulations}
Consider the signal strength field given by: 
$$c(\textbf{p},t)= -\log\left(1+(\textbf{p}-\textbf{p}^*(t))^\intercal(\textbf{p}-\textbf{p}^*(t))/2\right)$$
\begin{exmp}\thlabel{exmp:ex1}
For the first example, we let $\textbf{p}^*=0$, which corresponds to a static signal strength field that has a stationary source located at the origin. The initial conditions are taken as $\textbf{p}(0)=[-2,~-2,~6]^\intercal$ and $\textbf{R}=\textbf{I}_{3\times3}$. The parameters of the system are chosen as $\alpha = \frac{1}{8}$, $\omega=4\pi$, and $\mu=16\pi^2$.
\end{exmp}
\begin{exmp}\thlabel{exmp:ex2}
For the second example, we let  $\textbf{p}^*(t) = (2\sin(0.05 t),2\cos(0.05 t),2\cos(0.1 t))$, which corresponds to a time varying signal strength field that has a moving source located at $\textbf{p}^*(t)$. The initial conditions and parameters are the same as in \thref{exmp:ex1}.
\end{exmp}

The numerical simulation results are shown in Fig.\ref{fig:ex_costs} and Fig.\ref{fig:ex_trajs}. We observe how the original system closely approximates the RORA system.
\section{Conclusion}
In this note, we provide a novel use of the higher order averaging theorem in the trajectory approximation and stability analysis of a new class of high-amplitude, high-frequency, oscillatory systems with two distinct periodic time scales. In addition, we proposed a novel 3D source seeking algorithm for rigid bodies inspired by chemotaxis of sperm cells in certain marine animals, and established its practical stability as an application of recursive averaging. 

\section*{Acknowledgement} The authors like to thank Prof. Miroslav Krsti\'c for insightful suggestions. The authors also like to acknowledge the support of the NSF Grant CMMI-1846308. \\


\bibliography{3d_source_seeking}
\bibliographystyle{IEEEtran}

\end{document}